\documentclass{article}
\usepackage{graphicx} % Required for inserting images
\usepackage[utf8]{inputenc}
\usepackage{amsmath}
\usepackage{amsfonts}
\usepackage{amssymb,dsfont}
\usepackage{amsthm}
\usepackage{comment}
\usepackage{diagbox}
\usepackage{hyperref}

\newtheorem{theorem}{Theorem}[section]

\newtheorem{lemma}[theorem]{Lemma}
\theoremstyle{definition}

\theoremstyle{remark}
\newtheorem*{remark}{Remark}

\theoremstyle{definition}

\newcommand{\F}{\mathbb{F}}

\title{A lower bound on the number of bent squares}
\author{Jan Kristian Haugland \\ \texttt{\small{admin@neutreeko.net}}}

\begin{document}

\maketitle

\section{Introduction}
Bent functions are Boolean functions that are maximally nonlinear, and have been studied for decades. One out of several open problems is to find best possible estimates for $b_n$, the number of bent functions in $n$ variables where $n$ is a nonnegative even integer. Potapov \textit{et al.} \cite{potapov} noted that the best definitive lower bounds implied by earlier work were on the form
\begin{equation}\label{lowerbound}
\operatorname{log}_2 b_n \geq c n \cdot 2^{\frac{n}{2}} (1+o(1))
\end{equation}
with $c=\frac{1}{2}$, and showed that a construction by Baksova and Tarannikov \cite{baksova} yields $c=\frac{3}{4}$. The currently best \textit{upper} bound is $$\operatorname{log}_2 b_n \leq \frac{3}{8} \cdot 2^n (1+o(1))$$ by Potapov \cite{potapov2}.

Agievich \cite{agievich} introduced the representation of a bent function as a bent square, that is, a square matrix for which each row and each column is the Walsh spectrum of some Boolean function (confer Section 3). In Section 4, we apply this representation in order to prove that \eqref{lowerbound} holds with $c=1$. The result is not based on an explicit construction. Rather, we consider a set of potential building blocks, and show that a sufficiently large number of combinations must yield valid bent squares.

\section{Preliminaries}
A Boolean function in $n$ variables is a function $f:\F_2^n \to \F_2$. If each monomial in the representation of $f$ as a polynomial in $n$ variables over $\F_2$ is of degree 0 or 1, then $f$ is affine. Two Boolean functions $f$, $g$ in $n$ variables are extended affine equivalent, or EA-equivalent for short, if there exists an affine permutation $\pi:\F_2^n \to \F_2^n$ and an affine Boolean function $h$ in $n$ variables such that $g(x) = (f \circ \pi)(x) \oplus h(x)$.

A Boolean function $f$ in $n$ variables is called bent if the Walsh transform $$W_f(y) = \sum_{x \in \F_2^n}(-1)^{f(x) \oplus \langle x, y \rangle}$$ takes only the values $\pm 2^{\frac{n}{2}}$. Since $2^{\frac{n}{2}}$ must be an integer, bent functions in $n$ variables exist only for even $n$.

Let $H_n$ denote the $n$th Hadamard matrix of order $2^n$, defined iteratively by $H_0 = (1)$ and $$H_k=\left( \begin{matrix} H_{k-1} & H_{k-1} \\ H_{k-1} & -H_{k-1} \end{matrix} \right).$$ Notice that $H_n^2 = 2^n I$. With $[(-1)^{f(x)}]$ denoting a row vector with the values of $(-1)^{f(x)}$ as entries for all $x\in \F_2^n$ in lexicographical order, the Walsh spectrum of $f$ is given by $[(-1)^{f(x)}] H_n$. Thus, $f$ is bent if and only if each entry of its Walsh spectrum is $\pm 2^{\frac{n}{2}}$.

\section{Bent squares}
Let $f$ be a bent function in $n$ variables, and let $\mathcal{A}_f$ be the $2^{\frac{n}{2}} \times 2^{\frac{n}{2}}$ matrix for which the first row consists of the first $2^{\frac{n}{2}}$ entries of $[(-1)^{f(x)}]$, the second row consists of the next $2^{\frac{n}{2}}$ entries, and so on. The \textit{bent square} corresponding to $f$ is then given by $\mathcal{B}_f = \mathcal{A}_f H_{\frac{n}{2}}$.

Clearly, $2^{-\frac{n}{2}} \mathcal{B}_f H_{\frac{n}{2}}$ contains the entries of $[(-1)^{f(x)}]$, which are all $\pm 1$. Furthermore, $2^{-\frac{n}{2}} \mathcal{B}_f^\intercal H_{\frac{n}{2}}$ contains the entries of $2^{-\frac{n}{2}} [(-1)^{f(x)}] H_n$, which are also all $\pm 1$. It follows that each row and the transpose of each column of a bent square is the Walsh spectrum of some Boolean function in $\frac{n}{2}$ variables, and that any matrix with this property is a bent square corresponding to some bent function in $n$ variables.

As observed in \cite{agievich}, two types of admissible vectors for the rows and columns of a bent square are of particular interest due to their high sparsity.

\begin{itemize}
\item Type 1: The Walsh spectrum of an affine Boolean function in $\frac{n}{2}$ variables contains a single non-zero entry $\pm 2^{\frac{n}{2}}$.
\item Type 2: If $n \geq 4$ and we have a Boolean function in $\frac{n}{2}$ variables that is EA-equivalent to a single monomial of degree 2, then its Walsh spectrum contains four non-zero entries $\pm 2^{\frac{n}{2}-1}$. The indices of the non-zero entries form a 2-dimensional affine subspace of $\F_2^{\frac{n}{2}}$, and the number of positive entries is odd.
\end{itemize}
The number of $2^{\frac{n}{2}} \times 2^{\frac{n}{2}}$ bent squares whose rows and columns consist of vectors of type 1 is easily seen to be $2^{2^{\frac{n}{2}}}\left(2^{\frac{n}{2}}\right)!$, which implies \eqref{lowerbound} with $c=\frac{1}{2}$. We deviate from \cite{agievich} in the estimate of the number of $2^{\frac{n}{2}} \times 2^{\frac{n}{2}}$ bent squares whose rows and columns consist of vectors of type 2.

\section{A lower bound on $b_n$}
Suppose $n$ is an even integer $\geq 4$, and let $M_n$ denote the set of binary $2^{\frac{n}{2}-1} \times 2^{\frac{n}{2}-1}$ matrices for which exactly two entries in each row and each column are non-zero. Each element of $M_n$ has its rows and columns indexed by the elements of $\F_2^{\frac{n}{2}-1}$ in lexicographical order, and is assigned a vertical \textit{signature}, that is, the sequence of the XOR sums of the column indices of the non-zero elements of the rows in order. Similarly, the horizontal signature consists of the XOR sums of the row indices of the non-zero elements of the columns. An example is given in Table~\ref{matrix_example}. The non-zero elements are marked as bullet points. For instance, the column indices of the non-zero elements in the top row are 001 and 101, yielding the vertical signature value $001 \oplus 101 = 100$.

\begin{table}
\begin{tabular}{|c|c|c|c|c|c|c|c|c|c|}\hline
Index & 000 & 001 & 010 & 011 & 100 & 101 & 110 & 111 & V. sig. \\ \hline
000 & & \Large \textbullet & & & & \Large \textbullet & & & 100 \\ \hline
001 & & & \Large \textbullet & & & & \Large \textbullet & & 100 \\ \hline
010 & \Large \textbullet & & & \Large \textbullet & & & & & 011 \\ \hline
011 & & & & \Large \textbullet & & & \Large \textbullet & & 101 \\ \hline
100 & \Large \textbullet & & & & \Large \textbullet & & & & 100 \\ \hline
101 & & \Large \textbullet & & & & & & \Large \textbullet & 110 \\ \hline
110 & & & \Large \textbullet & & & \Large \textbullet & & & 111 \\ \hline
111 & & & & & \Large \textbullet & & & \Large \textbullet & 011 \\ \hline
H. sig. & 110 & 101 & 111 & 001 & 011 & 110 & 010 & 010 & \\ \hline
\end{tabular}
\caption{Example of a matrix in $M_8$, with its indices and signatures.}
\label{matrix_example}
\end{table}

The following result enables us to estimate the number of pairs of elements in $M_n$ with matching signatures.

\begin{lemma}\label{cauchy}
If $X$ and $Y$ are finite sets and $F:X \to Y$, then there are at least $\frac{|X|^2}{|Y|}$ ordered pairs $(x_1, x_2) \in X \times X$ such that $F(x_1)=F(x_2)$.
\end{lemma}

\begin{proof}
For $y \in Y$, let $u(y) = |F^{-1}(y)|$. Then there are $\sum_{y \in Y} u(y)^2$ ordered pairs $(x_1, x_2) \in X \times X$ such that $F(x_1) = F(x_2)$. Let $v(y) = 1$ for all $y \in Y$. By the Cauchy-Schwarz inequality $|\langle u, v \rangle | \leq ||u|| \cdot ||v||$, it follows that $$\sum_{y \in Y} u(y)^2 \geq \frac{\left(\sum_{y \in Y} u(y) \right)^2}{\sum_{y \in Y} v(y)^2} = \frac{|X|^2}{|Y|}.$$
\end{proof}

\begin{lemma}\label{signs}
For each element in $M_n$, there is at least one partition of the non-zero entries into two subsets such that each row and each column contains one entry from each subset.
\end{lemma}

\begin{proof}
Consider the graph with the non-zero entries as vertices and the unordered pairs of non-zero entries in the same row or the same column as edges. This is a finite 2-regular graph, and thus consists of disjoint cycles. Each cycle is of even length, as the edges are alternatingly horizontal and vertical. The graph is therefore bipartite.
\end{proof}

\begin{theorem}
If $n$ is an even integer $\geq 4$, then
\begin{equation}\label{mainresult}
\operatorname{log}_2 b_n \geq n \cdot2^{\frac{n}{2}} \left(1 + O\left(\frac{1}{n}\right) \right).
\end{equation}
\end{theorem}

\begin{proof}
Let $m_n$ denote the number of elements in $M_n$, and let $s_n$ denote the number of distinct vertical (or horizontal) signatures. By a result of Knuth \cite{knuth}, $$m_n \sim 2\sqrt{\pi} \left(\frac{2^{\frac{n}{2}-1}}{e}\right)^{2^{\frac{n}{2}}+\frac{1}{2}} \implies \operatorname{log}_2 m_n = \frac{1}{2}n \cdot 2^{\frac{n}{2}}\left(1 + O\left( \frac{1}{n} \right)\right).$$ Each of the first $2^{\frac{n}{2}-1}-1$ XOR sums in a vertical or horizontal signature can take $2^{\frac{n}{2}-1}-1$ different values, and the last entry is the XOR sum of all the previous ones due to the cycle structure mentioned in the proof of Lemma~\ref{signs}. Therefore, $$s_n \leq \left(2^{\frac{n}{2}-1}-1\right)^{2^{\frac{n}{2}-1}-1} \implies \operatorname{log}_2 s_n \leq \frac{1}{4}n \cdot 2^{\frac{n}{2}}\left(1 + O\left( \frac{1}{n} \right) \right).$$ By Lemma~\ref{cauchy}, there are at least $$\frac{m_n^2}{s_n}$$ ordered pairs $(A, B) \in M_n \times M_n$ for which $A$ and $B$ have the same vertical signature. Those pairs can have at most $s_n^2$ distinct combined horizontal signatures. By Lemma~\ref{cauchy} again, it follows that there are at least $$\frac{m_n^4}{s_n^4}$$ ordered pairs of pairs $((A, B), (C, D)) \in (M_n \times M_n) \times (M_n \times M_n)$ such that $A$ and $B$ have the same vertical signature, $C$ and $D$ have the same vertical signature, $A$ and $C$ have the same horizontal signature, and $B$ and $D$ have the same horizontal signature. For each such quadruple $(A, B, C, D)$, $$\left(\begin{matrix}A & B \\ C & D\end{matrix}\right)$$ is a distinct $2^{\frac{n}{2}} \times 2^{\frac{n}{2}}$ matrix such that the indices of the non-zero elements in each row and each column form a 2-dimensional affine subspace of $\F_2^{\frac{n}{2}}$.

For each of these matrices, there are at least 32 ways to replace the non-zero entries by $\pm 2^{\frac{n}{2}-1}$ so that we end up with an odd number of positive entries in each row and each column. For example, we can let the entries of $A$ be all positive or all negative, let $B$ and $C$ have one positive and one negative entry in each row and each column, and let the entries of $D$ be all positive or all negative. This accounts for at least 16 ways, by Lemma~\ref{signs}. In addition, we can let $A$ and $D$ have one positive and one negative entry in each row and each column, while the entries of $B$ and of $C$ are either all positive or all negative.

Thus, counting bent squares whose rows and columns consist of vectors of type 2 only, we have
\begin{equation}\label{specific}
b_n \geq 32 \frac{m_n^4}{s_n^4}
\end{equation}
which implies \eqref{mainresult}.
\end{proof}

\begin{remark}
The estimate \eqref{specific} is not intended to be best possible; there are several ways to improve it. However, we are not aware of any refinement of the argument that would do more than altering the implied constant in the term $O\left( \frac{1}{n} \right)$ in \eqref{mainresult}.
\end{remark}

\end{document}